\documentclass[12p]{amsart}
\usepackage{amssymb}
\usepackage{amsmath}
\usepackage{amsfonts}
\usepackage{geometry}
\usepackage{graphicx}
\usepackage{mathrsfs,amssymb}

\usepackage{hyperref}
\usepackage{cleveref}

\theoremstyle{plain}

\newtheorem{conjecture}{Conjecture}
\newtheorem{corollary}{Corollary}

\newtheorem{lemma}{Lemma}

\newtheorem{proposition}{Proposition}
\newtheorem{remark}{Remark}

\newtheorem{theorem}{Theorem}
\numberwithin{equation}{section}

\newcommand{\ls}{\lesssim}
\newcommand{\gs}{\gtrsim}

\begin{document}

\title[]{Interior estimates for the eigenfunctions of the fractional Laplacian on a bounded euclidean domain  }

\author{Xiaoqi Huang}
\author{Yannick Sire}
\author{Cheng Zhang}

\begin{abstract}
This paper is devoted to interior, i.e. away from the boundary, estimates for eigenfunctions of the fractional Laplacian in an Euclidean domain of $\mathbb R^d$.
\end{abstract}
\maketitle
\tableofcontents

\section{Introduction}

This paper is devoted to the investigation of the following eigenvalue problem

\begin{equation}\label{eigen}
 \begin{aligned}
     (-\Delta)^{\alpha/2}e_\lambda&=\lambda e_\lambda  , \qquad  x\in \Omega \\
      e_\lambda&=0, \qquad x\in \mathbb{R}^d\setminus \Omega .
 \end{aligned}
 \end{equation}
  Here we denoted by $(-\Delta)^{\alpha/2}$ for $0<\alpha<2$ the Fourier multiplier of symbol $|\xi|^\alpha$ and by $\Omega$ a $C^{1,1}$ bounded domain in $\mathbb{R}^d$ for $d \geq 1$. The fractional laplacian is also defined pointwise by the principal value integral:
\[(-\Delta)^{\alpha/2}u(x)= c_{d,\alpha} \hspace{1mm} p.v. \int_{\mathbb{R}^d} \frac{u(x)-u(y)}{|x-y|^{d+\alpha}},\ x\in \mathbb{R}^d\]where $c_{d,\alpha}=\frac{2^{\alpha} \Gamma(\frac{d+\alpha}{2})}{\pi^{d / 2}|\Gamma(-\frac{\alpha}{2})|}
 \hspace{1mm} $ is a normalizing constant. We are interested in interior bounds for $e_\lambda$ in $L^p$ in terms of $\lambda$ for the range $0<\alpha<2$ and $p \geq 2$.  The previous spectral problem involves the so-called {\sl restricted} fractional Laplacian, i.e. the Fourier multiplier $|\xi|^\alpha$ whose domain is a Sobolev space of function vanishing outside $\Omega$.

 Well-known results (see for instance \cite{getoor,grubb,bookStable}) ensure that the spectrum is discrete and nonnegative, and the eigenfunctions are smooth inside the domain. We refer the reader to the survey \cite{frank} for several results and open problems related to eigenvalues and eigenfunctions of the restricted fractional Laplacian.

 It is well-known that on bounded domains (or manifolds with boundary) the issue of obtaining global $L^p$ bounds is a difficult task (see \cite{smithSogge} for instance). For the Laplace-Beltrami operator on closed manifolds, the first $L^p$ bounds were obtained by Sogge \cite{soggeJFA}.

 The case of the fractional Laplacian is largely open as far as bounds on eigenfunctions or spectral projectors are concerned. However, motivated by issues in quantum mechanics and statistical physics, various bounds for the eigenvalues have been investigated. We refer the reader to the nice survey \cite{frank} for an updated account on these spectral issues. See also \cite{stos}, \cite{kwasnicki}, \cite{DKK17} for the study of eigenfunctions on the intervals or unit balls. The book \cite{bookStable} gives a great account on the potential analysis of stable processes, the class of Levy processes in which the fractional laplacian is the simplest infinitesimal generator (see \cite{bertoin}). In particular, in Chapter $4$, Kulczycki raises explicitly several basic, yet open, questions. The present paper is a contribution towards this program.

 As a first step towards understanding the whole picture, we consider here interior estimates, i.e. $L^p(K)$ bounds for $K\subset\subset\Omega$ and $p \geq 2$. Hence  the regularity of $\Omega$ will not play a crucial role, except for some heat kernel bounds. The whole point of our study will be to reduce the estimate to a commutator estimate and get bounds for it. Despite the fact we are considering interior estimates, the problem being nonlocal in nature, the fact that the equation is set in a bounded domain (and not the whole of $\mathbb R^d$) induces several unavoidable difficulties. There are several ways to define the fractional on a bounded domain. The one we consider here is the so-called {\sl restricted fractional Laplacian} and is the first natural version one can imagine. The other one is called the {\sl spectral fractional Laplacian}, which is given by the spectral theorem and the eigenfunctions are the same as the one of the Dirichlet Laplacian. A last one, arising in probabilities, is the so-called {\sl regional fractional Laplacian} and is given by
 $$
p.v. \int_{\Omega} \frac{u(x)-u(y)}{|x-y|^{d+\alpha}}\,dy.
 $$
 Heat kernel bounds (see \cite{censored}) are known also for this operator; however this operator does not coincide locally with the restricted fractional Laplacian and our approach does not allow to deal with it. We refer the reader to \cite{BSV, BFV} for accounts on the various fractional Laplacians in domains.

Throughout  this paper, we use the following notations.  $A\ls B$ ($A\gs B$) means $A\le CB$ ($A\ge cB$) for some positive constants $C,\ c$ independent of $\lambda$. $A\approx B$ means $A\ls B$ and $A\gs B$. The constants may depend on the domain $\Omega$, the compact set $K$, and the fixed parameters $d,\ \alpha,\ p$, and it is possible to find out the explicit dependence on them. The norm $\|\cdot\|_p$ means the $L^p$ norm in the whole Euclidean space $\mathbb{R}^d$.

 \subsection*{Main results}

 Our main result is the following.
 \begin{theorem}\label{main}
 Let $\Omega\subset \mathbb{R}^d$ be a bounded $C^{1,1}$ domain. Let $K\subset\subset \Omega$ be a compact set. Consider the following cases:
 \begin{itemize}
 \item Either $d=1$, $\frac14\le \alpha<2$, $2\le p\le \infty$
 \item Or $d=1$, $0<\alpha<\frac14$, $2\le p\le \frac{2(1-2\alpha)}{1-4\alpha}$
 \item Or $d\ge2$, $\frac12< \alpha<2$, $2\le p\le \infty$
 \end{itemize}

 Then for $\lambda>1$ and $e_\lambda$ satisfying \eqref{eigen} we have
\[\|e_\lambda\|_{L^p(K)}\ls \lambda^{\sigma(d,\alpha,p)}\|e_\lambda\|_2\]
where
\[\sigma(d,\alpha,p)=\max\Big\{\frac{d-1}{2\alpha}(\frac12-\frac1p),\frac{d-1}{2\alpha}-\frac{d}{\alpha p}\Big\}.\]
The  constant here is independent of $\lambda$.
 \end{theorem}
 In other words, if $p_c=\frac{2d+2}{d-1}$, then
\[\|e_\lambda\|_{L^p(K)}\ls \lambda^{\frac{d-1}{2\alpha}(\frac12-\frac1p)}\|e_\lambda\|_2,\ 2\le p\le p_c\]
\[\|e_\lambda\|_{L^p(K)}\ls \lambda^{\frac{d-1}{2\alpha}-\frac{d}{\alpha p}}\|e_\lambda\|_2,\ p_c\le p\le \infty.\]
In particular, it gives the uniform bound in dimension $d=1$ for $\lambda>1$
\[\|e_\lambda\|_{L^p(K)}\ls\|e_\lambda\|_2, \ 2\le p\le \infty.\]
These results agree with Sogge's $L^p$ estimates when $\alpha=2$ (see \cite{soggeJFA}) and a conjecture by Kwa\'snicki when $d=1$ (see  \cite{kwasnicki} and \cite{stos}). Indeed, Kwa\'snicki conjectured that for $0<\alpha<2$, the one dimensional eigenfunctions are uniformly bounded on $\Omega$:
\[\|e_\lambda\|_{L^\infty(\Omega)}\ls\|e_\lambda\|_2\]
where the constant is independent of $\lambda$.
Kwa\'snicki proved it for $\frac12\le \alpha<2$ (see also \cite{stos} for $\alpha=1$). Our Theorem \ref{main} gives uniform bounds on the interior $L^p$ estimates when $0< \alpha<2$, which provides evidence for this conjecture in the whole range of $\alpha$. Furthermore, it seems natural to make the following conjecture:
\begin{conjecture}\label{conj}Let $d\ge1$, $0<\alpha<2$, $2\le p\le \infty$. Let $\Omega\subset \mathbb{R}^d$ be a bounded $C^{1,1}$ domain. Let $K\subset\subset \Omega$ be a compact set. Then for $\lambda>1$ and $e_\lambda$ satisfying \eqref{eigen} we have
\[\|e_\lambda\|_{L^p(K)}\ls \lambda^{\sigma(d,\alpha,p)}\|e_\lambda\|_2\]
where
\[\sigma(d,\alpha,p)=\max\Big\{\frac{d-1}{2\alpha}(\frac12-\frac1p),\frac{d-1}{2\alpha}-\frac{d}{\alpha p}\Big\}.\]
The  constant here is independent of $\lambda$.
\end{conjecture}

Theorem \ref{main} answers this conjecture except in the following two cases:

\begin{enumerate}
\item $d=1$, $0<\alpha<\frac14$, $\frac{2(1-2\alpha)}{1-4\alpha}< p\le \infty$\\
\item $d\ge2$, $0< \alpha\le \frac12$, $2< p\le \infty$.
\end{enumerate}
In these ranges of parameters, we are able to prove some $L^p$ estimates weaker than the conjecture bounds. See Proposition \ref{prop2}, Section 4 and Section 5. In particular, when $\alpha=\frac12$, our $L^p$ estimates agree with the conjecture bounds up to some log factors.
\begin{theorem}\label{main2}
Let $d\ge2$, $\alpha=\frac12$ and $2\le p\le \infty$. Let $\Omega\subset \mathbb{R}^d$ be a bounded $C^{1,1}$ domain. Let $K\subset\subset \Omega$ be a compact set.
Then for $\lambda>1$ and $e_\lambda$ satisfying \eqref{eigen} we have
\[\|e_\lambda\|_{L^p(K)}\ls \lambda^{\sigma(d,\alpha,p)}(\log\lambda)^{\gamma(d,p)}\|e_\lambda\|_2\]
where $\gamma(d,p)=\min\{\frac14,\frac{d+1}4(\frac12-\frac1p)\}$.
The  constant here is independent of $\lambda$.
\end{theorem}

When $\frac12<\alpha<2$, Theorem \ref{main} is actually a corollary of the following result.

\begin{theorem}\label{th:commut}
Let $\beta $ be a smooth cut-off function such that $\beta=1 $ in $K$, and $\beta=0 $ outside a small neighbourhood of $K$ which is contained in $\Omega$, with $K\subset\subset  \Omega$. Assume that $e_\lambda$ is an eigenfunction satisfying \eqref{eigen}. If the following commutator estimate holds
\begin{equation}\label{comm}
\|[\beta,(-\Delta)^{\alpha/2}]e_\lambda\|_{2}\ls \lambda^{1-\frac{1}{\alpha}}\|e_\lambda\|_{2},\ \ \lambda>1,\end{equation}
then Conjecture \ref{conj} holds.
\end{theorem}

Except in the range $\frac12<\alpha<2$, we are not able to prove the desired commutator estimate \eqref{comm}, but a weaker one. In particular, when $\alpha=\frac12$, we can prove   \[\|[\beta,(-\Delta)^{\alpha/2}]e_\lambda\|_{2}\ls \lambda^{-1}(\log\lambda)^\frac12\|e_\lambda\|_{2}.\]

As far as global $L^p$ bounds are concerned (i.e in all of $\Omega$), we conjecture that the phenomenon observed in \cite{smithSogge}, namely concentration of eigenfunctions close to the boundary for some values of $p$, does not hold in our case. This claim is supported by the nonlocal effects of the fractional Laplacian, in the sense that the operator {\sl does not see} the boundary $\partial \Omega$. This phenomenon has been observed in many instances (see e.g. \cite{caffaSire} for an account). From this point of view, nonlocal operators are easier to study than local ones, as far as boundary effects are concerned.

The present paper opens the way to provide a refined study of eigenfunction estimates for problems with boundary conditions involving the eigenfunction itself, like the {\sl Steklov} problem. This latter is intrinsically a nonlocal problem and we plan to address this issue in a forthcoming work \cite{HSZ}.

The paper is organized  as follows. In Section 2, in order to use known results for the fractional Laplacian in the {\sl whole} space, we consider a standard cut-off function applied to the eigenfunction $u$ and re-write the eigenvalue problem in a suitable form. Using the heat kernel method and resolvent estimates, we reduce the whole study to a commutator estimate.

In Section 3, we prove the commutator estimate when $\frac12<\alpha<2$ by using heat kernel estimates. This proves Theorem \ref{main} for $\alpha$ in this range. We also obtain some $L^p$ estimates weaker than the conjecture bounds when $0<\alpha\le \frac12$. In particular, when $\alpha=\frac12$, our $L^p$ estimates agree with the conjecture bounds up to some log factors.

In Section 4, we study the eigenfunction estimates when $0<\alpha<\frac12$. We decompose the range of $\alpha$ into small intervals $[\frac1{N+1},\frac1N)$, $N\ge2$. Then we prove a commutator estimate in each of them, and improve the eigenfunction estimates when $0<\alpha<\frac12$.

In Section 5, we focus on 1 dimensional eigenfunction estimates when $0<\alpha\le \frac12$, and complete the proof of Theorem \ref{main}. We will combine the previous ideas and the approximation method to prove better $L^\infty$ and $L^p$ estimates.

\section{Reduction to a commutator estimate}
We will use the heat kernel method and resolvent estimates to reduce the problem to a commutator estimate. This is the main idea behind Theorem \ref{main}.

Choose a compact set $K_1$ such that $K\subset\subset K_1\subset \Omega$ and a smooth cut-off function $\beta$ such that $\beta=1 $ in $K_1$, and $\beta=0 $ outside a small neighbourhood of $K_1$ which is contained in $\Omega$.

Denoting  $v=\beta e_\lambda$, one has
\begin{equation}
 \begin{aligned}
     (-\Delta)^{\alpha/2}v&=\beta(-\Delta)^{\alpha/2}e_\lambda-[\beta,(-\Delta)^{\alpha/2}]e_\lambda \\
      &=\lambda v-[\beta,(-\Delta)^{\alpha/2}]e_\lambda
 \end{aligned}
 \end{equation}
Fix $z=\lambda+\lambda^{1-\frac1\alpha}i$, $\lambda\gg1$.  So we obtain
\[ v=((-\Delta)^{\alpha/2}-z)^{-1}((\lambda-z)v-[\beta,(-\Delta)^{\alpha/2}]e_\lambda)\]
Note that $0<\arg(z)<\pi \alpha/4$ for large $\lambda$. We may use the following identity for this choice of $z$ and $0<\alpha<2$ (see \cite{MSbook}, page 118, (5.28)):
\begin{align*}((-\Delta)^{\alpha/2}-z)^{-1}&=\frac{z^{(2-\alpha)/\alpha}}{\alpha/2}(-\Delta-z^{2/\alpha})^{-1}+\frac{\sin(\pi \alpha/2)}{\pi}\int_0^\infty\frac{\tau^{\alpha/2}(\tau-\Delta)^{-1}}{\tau^\alpha-2z\tau^{\alpha/2}\cos(\pi \alpha/2)+z^2}d\tau\\
&:=T_0+R\ .\end{align*}
Now we need to estimate the norms of $T_0$ and $R$. Fix $p_c=\frac{2d+2}{d-1}$. In particular, we set $p_c=\infty$ when $d=1$. We need the following sharp resolvent estimates for the standard Laplacian in $\mathbb{R}^d$. It follows directly from \cite{KwonLee}, Theorem 1.4 and Remark 2.

\begin{lemma}\label{lemma1}Let $d\ge1$ and $z\in\mathbb{C}\setminus [0,\infty)$. Then for $2\le p\le p_c$
\[\|(-\Delta-z)^{-1}\|_{2\to p}\approx |z|^{-1+\frac{d}2(\frac12-\frac1p)}dist(\tfrac{z}{|z|},[0,\infty))^{\frac{d+1}2(\frac12-\frac1p)-1}.\]
Moreover, for $p_c\le p\le \infty$,
\[\|(-\Delta-z)^{-1}\|_{p\to p}\approx |z|^{-1}dist(\tfrac{z}{|z|},[0,\infty))^{\frac{d+1}2-\frac{d}p}.\]
The constants here only depend on $d$ and $p$.
\end{lemma}

Then we have for $d\ge1$ and $2\le p\le p_c$
\begin{equation}\label{T0}\|T_0\|_{2\to p}\ls |z|^{\frac{d}{\alpha}(\frac12-\frac1p)-1}dist((\tfrac{z}{|z|})^{2/\alpha},[0,\infty))^{\frac{d+1}2(\frac12-\frac1p)-1}\approx \lambda^{\frac{d-1}{2\alpha}(\frac12-\frac1p)+\frac{1-\alpha}\alpha}.
\end{equation}
For $\tau>0$ and $d\ge1$, the kernel of $(\tau-\Delta)^{-1}$ satisfies
\[(\tau-\Delta)^{-1}(x,y)=\tau^{\frac{d-2}4}|x-y|^{\frac{2-d}2}K_{\frac{d-2}2}(\tau^{1/2}|x-y|)\]
where $K_m(r)$ are the modified Bessel functions of the second kind. See e.g. \cite{KRS89}, \cite{GS64}. When $0<r\le1$,
\[|K_{\frac{d-2}2}(r)|\le \begin{cases}C_d r^{-\frac{d-2}2},\ d\ge3\\
C|\log(r/2)|,\ d=2\\
C r^{-1/2},\ d=1
\end{cases}\]
and when $r\ge1$
\[|K_{\frac{d-2}2}(r)|\le C_d r^{-1/2}e^{-r},\ d\ge1.\]
These estimates imply that when $\tau^\frac12|x-y|\le1$
\[|(\tau-\Delta)^{-1}(x,y)|\ls \begin{cases}
  |x-y|^{2-d},\ d\ge3\\
  |\log(\tau^\frac12|x-y|/2)|,\ d=2\\
  \tau^{-1/2},\ d=1
\end{cases} \]
and when $\tau^\frac12|x-y|\ge1$
\[|(\tau-\Delta)^{-1}(x,y)|\ls\tau^{\frac{d-3}4}|x-y|^{\frac{1-d}2}e^{-\tau^{1/2}|x-y|},\ d\ge1.\]
Note that
\[\tau^\alpha-2z\tau^{\alpha/2}\cos(\pi \alpha/2)+z^2=(\tau^{\alpha/2}-ze^{i\pi \alpha/2})(\tau^{\alpha/2}-ze^{-i\pi \alpha/2})\]
Recall that $0<\arg(z)<\pi \alpha/4$. So we have $|\arg(ze^{\pm i\pi \alpha/2})|>\pi \alpha/4$. Then
\[|\tau^{\alpha/2}-ze^{\pm i\pi \alpha/2}|\approx \tau^{\alpha/2}+|z|\]
which implies
\[|\tau^\alpha-2z\tau^{\alpha/2}\cos(\pi \alpha/2)+z^2|\approx \tau^{\alpha}+|z|^2.\]
Combing this estimate with the bounds on the kernel $(\tau-\Delta)^{-1}(x,y)$, we are able to compute the kernel of $R$. We may use $|\tau^\alpha-2z\tau^{\alpha/2}\cos(\pi \alpha/2)+z^2|\gs \tau^{\alpha}$ to obtain
\begin{equation}\label{R1}|R(x,y)|\ls |x-y|^{\alpha-d},\ d\ge1.
\end{equation}
Moreover, we may exploit $|\tau^\alpha-2z\tau^{\alpha/2}\cos(\pi \alpha/2)+z^2|\gs |z|^2$ to get
\begin{equation}\label{R2}
  |R(x,y)|\ls |z|^{-2}|x-y|^{-\alpha-d},\ d\ge1.
\end{equation}
For optimization, we will use \eqref{R1} when $|x-y|\le|z|^{-1/\alpha}$ and use \eqref{R2} when $|x-y|\ge|z|^{-1/\alpha}$.
Thus, by Young's inequality (or Hardy-Littlewood-Sobolev inequality at $p_o=\frac{2d}{d-2\alpha}$), we have
\begin{equation}\label{R}
  \|R\|_{2\to p}\ls |z|^{\frac{d}{\alpha}(\frac12-\frac1p)-1}\approx \lambda^{\frac{d}{\alpha}(\frac12-\frac1p)-1},\ p\le p_o.
\end{equation}

Next, we will use the heat kernel method. Let $P_t^\Omega$ be the heat semigroup for the Dirichlet fractional Laplacian $-(-\Delta)^{\alpha/2}|_\Omega$. Then
\[P_t^\Omega e_\lambda=e^{-t\lambda}e_\lambda\]
which gives $P_t^\Omega e_\lambda=e^{-1}e_\lambda$ when $t=\lambda^{-1}$. We always fix $t=\lambda^{-1}$. Let $p_\Omega(t,x,y)$ be the heat kernel of the Dirichlet fractional Laplacian $-(-\Delta)^{\alpha/2}|_\Omega$. We will use the following two-sided heat kernel estimates from \cite{chenJEMS}, Theorem 1.1.
\begin{lemma} \label{heat}Let $0<\alpha<2$. Let $\Omega$ be a $C^{1,1}$ open subset of $\mathbb{R}^d$ with $d\ge1$ and $\delta_\Omega(x)$ the Euclidean distance between $x$ and $\Omega^c$. Then
\[p_\Omega(t,x,y)\approx \Big(1\wedge \frac{\delta_\Omega(x)^{\alpha/2}}{\sqrt t}\Big)\Big(1\wedge \frac{\delta_\Omega(y)^{\alpha/2}}{\sqrt t}\Big)\Big(t^{-d/\alpha}\wedge \frac{t}{|x-y|^{d+\alpha}}\Big), \ x,y\in\Omega,\ t\in(0,1].\]
\end{lemma}

Then for $x\in K\subset\subset K_1$,
\begin{align*}
  |e_\lambda(x)|&\ls |\int_{K_1}p_\Omega(t,x,y)e_\lambda(y)dy|+|\int_{\Omega\setminus K_1}p_\Omega(t,x,y)e_\lambda(y)dy|\\
  &\ls |\int_{K_1}p_\Omega(t,x,y)v(y)dy|+\|e_\lambda\|_2\\
  &:=|Qv(x)|+\|e_\lambda\|_2
\end{align*}
Let $w=(\lambda-z)v-[\beta,(-\Delta)^{\alpha/2}]e_\lambda$. Since $v=(T_0+R)w$, we get for $2\le p\le \infty$
\begin{equation}\label{uest}
  \|e_\lambda\|_{L^p(K)}\ls \|Qv\|_p+\|e_\lambda\|_2\ls (\|QT_0\|_{2\to p}+\|QR\|_{2\to p})\|w\|_2+\|e_\lambda\|_2.
\end{equation}
By the heat kernel estimate and Young's inequality, we have
\begin{equation}\label{Q}
  \|Q\|_{p\to q}\ls \lambda^{\frac{d}{\alpha}(\frac1p-\frac1q)},\ 2\le p\le q\le \infty.
\end{equation}
Combining this with \eqref{R} we get
\begin{equation}\label{QR}
  \|QR\|_{2\to p}\ls \lambda^{\frac{d}{\alpha}(\frac12-\frac1p)-1},\ 2\le p\le\infty.
\end{equation}
We may observe that this bound \eqref{QR} is independent of the choice of $z$ whenever $|z|\approx \lambda$. Moreover, by using \eqref{T0} and \eqref{Q}, we have
\begin{equation}\label{QT0}
  \|QT_0\|_{2\to p}\ls \begin{cases}\lambda^{\frac{d-1}{2\alpha}(\frac12-\frac1p)+\frac{1-\alpha}\alpha},\ 2\le p\le p_c\\
\lambda^{\frac{d-1}{2\alpha}-\frac{d}{\alpha p}+\frac{1-\alpha}\alpha},\ p_c\le p\le \infty.
  \end{cases}\end{equation}
Note that
\[\frac{d-1}{2\alpha}(\frac12-\frac1p)+\frac{1-\alpha}\alpha>\frac{d}{\alpha}(\frac12-\frac1p)-1,\ 2\le p\le p_c\]
\[\frac{d-1}{2\alpha}-\frac{d}{\alpha p}+\frac{1-\alpha}\alpha>\frac{d}{\alpha}(\frac12-\frac1p)-1,\ p_c\le p\le \infty.\]
Thus $\|QT_0\|_{2\to p}$ is the main term in \eqref{uest}. Therefore, if one can prove the following commutator estimate for $0<\alpha<2$ (namely \eqref{comm})
\begin{equation}
\|[\beta,(-\Delta)^{\alpha/2}]e_\lambda\|_{2}\ls \lambda^{1-\frac{1}{\alpha}}\|e_\lambda\|_{2}\end{equation}
then
\[\|w\|_2\ls \lambda^{\frac{\alpha-1}\alpha}\|e_\lambda\|_2\]
and Conjecture \ref{conj} follows.
 \section{Proof of the eigenfunction estimates}
 In this section, we prove Theorem \ref{main} for $\frac12<\alpha<2$. Indeed, we will prove \eqref{comm} for $\frac12<\alpha<2$, which implies the eigenfunction estimates for $\alpha$ in this range. Moreover, we will prove that \eqref{comm} is also true at $\alpha=\frac12$, up to a log factor. The proof is divided in two steps, each one involving a different argument.

\subsection{Proof when $1\le \alpha<2$ }
We need the following lemma, which is Proposition 4.2 in \cite{taylor}.
\begin{lemma}\label{lemma3}
  Let $1<p<\infty$, $m\ge0$, $d\ge1$. Given $P\in OP\mathcal{B}S_{1,1}^m$,
  \[\|[P, f] u\|_{H^{s, p}} \le C\|f\|_{H^{\sigma, p}}\|u\|_{H^{s+m-1, p}}\]
  provided
  \[\sigma>\frac{d}{p}+1, \quad s \geq 0, \quad s+m \le \sigma.\]
  Here $f$ and $u$ are defined on $\mathbb{R}^d$.
\end{lemma}
By definition, an element $p(x,D)$ of $OPS_{1,\delta}^m$ has symbol $p(x,\xi)$, satisfying
\[\left|D_{x}^{\beta} D_{\xi}^{\alpha} p(x, \xi)\right| \leq C_{\alpha \beta}(1+|\xi|)^{m-|\alpha|+\delta|\beta|}.\]
The class $O P \mathcal{B} S_{1,1}^{m}$ consists of operators with symbol $p(x, \xi) \in S_{1,1}^{m}$, satisfying
\[\operatorname{supp} \hat{p}(\eta, \xi) \subset\{(\eta, \xi) :|\eta| \leq \rho|\xi|\}\]
for some $\rho<1$. This class was introduced by \cite{Mey81} and contains the paradifferential operators introduced in \cite{Bon81}. We note that $OP\mathcal{B}S^m_{1,1}$ contains all the operator classes $OPS^m_{1,\delta}$ $\delta < 1$, at least modulo smoothing operators.
\begin{proposition}\label{prop1}
  For any $0<\alpha<2$, we have
  \begin{equation}
    \|[\beta,(-\Delta)^{\alpha/2}]u\|_{2}\ls\|u\|_{H^{\alpha-1}}.
  \end{equation}
\end{proposition}
\begin{proof}
  Let $L=(I-\Delta)^{\alpha/2}$ be a pseudo differential operator of order $\alpha$ (i.e. $L\in OPS_{1,0}^\alpha$). Hence $L\in OP\mathcal{B}S_{1,1}^\alpha$. By Lemma \ref{lemma3}, we have the following commutator estimate:
 \begin{equation}\label{taylor}
\|[L, \beta] u\|_{2} \leq C\|\beta\|_{H^{\sigma}}\|u\|_{H^{\alpha-1}}
\end{equation}
given:
\begin{equation}
\sigma>\frac{d}{2}+1, \quad 0\le \alpha \leq \sigma.
\end{equation}
 Since $\beta\in C_0^\infty$, we may choose $\sigma$ sufficiently large. So it suffices to show
  \begin{equation}\label{cha}
    \|[L-(-\Delta)^{\alpha/2}, \beta]u\|_{L^2}\ls\|u\|_{H^{\alpha-1}}.
  \end{equation}Note that $L-(-\Delta)^{\alpha/2}$ is a Fourier mutiplier $$(1+|\xi|^2)^{\frac{\alpha}{2}}-|\xi|^\alpha= \frac{(1+|\xi|^2)^{\frac{\alpha}{2}}-|\xi|^\alpha} {(1+|\xi|^2)^{\frac{\alpha-1}{2}}}(1+|\xi|^2)^{\frac{\alpha-1}{2}}=m(\xi)(1+|\xi|^2)^{\frac{\alpha-1}{2}}$$
A direct calculation shows that $m(\xi) \in L^\infty$. Then by Plancherel theorem
\begin{equation}\label{pl}
  \|(L-(-\Delta)^{\alpha/2})u\|_{2}\ls\|(1+|\xi|^2)^{\frac{\alpha-1}2}\hat u(\xi)\|_2=\|u\|_{H^{\alpha-1}}
\end{equation}
Then to prove \eqref{cha}, we only need to show
  \[\|\beta u \|_{H^{\alpha-1}}\ls \|u\|_{H^{\alpha-1}}.\]
Let $a=\alpha-1$. \begin{align*}
  \|\beta u \|_{H^{a}}^2&\ls\int |\hat\beta*\hat u(\xi)|^2(1+|\xi|^2)^ad\xi\\
  &\ls \int|\int\hat\beta(\eta)\hat u(\xi-\eta)d\eta|^2(1+|\xi|^2)^ad\xi\\
  &\ls \int\int|\hat\beta(\eta)||\hat u(\xi-\eta)|^2d\eta(1+|\xi|^2)^ad\xi\\
  &= \int|\hat\beta(\eta)|\int|\hat u(\xi)|^2(1+|\xi+\eta|^2)^ad\xi d\eta\\
\end{align*}
When $a\ge0$, we have
\[(1+|\xi+\eta|^2)^a\ls (1+|\xi|^2)^a(1+|\eta|^2)^a,\]
and when $a<0$, we have
\[(1+|\xi+\eta|^2)^a\ls (1+|\xi|^2)^a(1+|\eta|^2)^{-a}.\]
Thus the desired inequality $\|\beta u \|_{H^{a}}\ls \|u\|_{H^{a}}$ follows from the fact that $\hat\beta$ is a Schwartz function.
\end{proof}
We need the gradient estimate for the heat kernel to estimate $\|e_\lambda\|_{H^{\alpha-1}}$. It follows from \cite{ryznar}, Theorem 1.1 and Example 5.1.
\begin{lemma} \label{grad}Let $0<\alpha<2$. Let $\Omega$ be any open, nonempty set in $\mathbb{R}^d$. Let $p_\Omega(t,x,y)$ be the heat kernel of the Dirichlet fractional Laplacian $-(-\Delta)^{\alpha/2}|_\Omega$. Then
\[|\nabla_x p_\Omega(t,x,y)|\le \frac{c}{\delta_\Omega(x)\wedge t^{1/\alpha}}p_\Omega(t,x,y),\ x,y\in\Omega,\ t\in(0,1]\]
where $c=c(d,\alpha)$ and $\delta_\Omega(x)=dist(x,\Omega^c)$.
\end{lemma}

\begin{theorem}\label{thm1} If $e_\lambda$ satisfies \eqref{eigen}, then for $1\le \alpha<2$ we have
  \[\|e_\lambda\|_{H^{\alpha-1}}\ls \lambda^{1-\frac1\alpha}\|e_\lambda\|_2.\]
\end{theorem}
\begin{proof}
  It is clear when  $\alpha=1$. Assume $1<\alpha<2$. By Holder's inequality,
  \[\|e_\lambda\|_{H^{\alpha-1}}\le \|e_\lambda\|_2^{2-\alpha}\|e_\lambda\|_{H^1}^{\alpha-1}.\]
  Note that
  \[\|e_\lambda\|_{H^1}^2= \|e_\lambda\|_2^2+\|\nabla e_\lambda\|_2^2.\]
  It suffices to prove the following gradient estimate for $1<\alpha<2$
  \begin{equation}\label{gradu}
    \|\nabla e_\lambda\|_2\ls \lambda^{\frac1\alpha}\|e_\lambda\|_2.
  \end{equation}
  Since $P_t^\Omega e_\lambda=e^{-t\lambda}e_\lambda$, we have
  \[e^{-t\lambda}\nabla e_\lambda(x)=\int \nabla_xp_\Omega(t,x,y)e_\lambda(y)dy.\]
If we set $t=\lambda^{-1}$, then to estimate $\|\nabla e_\lambda\|_2$ we only need to compute the $L^2\to L^2$ norm of the operator associated with the kernel $\nabla_xp_\Omega(t,x,y)$. We may use Schur's test to compute the operator norm. Let
\[q(x)=\Big(1\wedge \frac{\delta_\Omega(x)}{t^{1/\alpha}}\Big)^{-1/2}.\]
Then using Lemma \ref{heat} and Lemma \ref{grad}, we claim that
\[\int_\Omega |\nabla_xp_\Omega(t,x,y)| dy\ls t^{-1/\alpha}q(x)\]
\[\int_\Omega |\nabla_xp_\Omega(t,x,y)|q(x)dx\ls t^{-1/\alpha}\]
when $1<\alpha<2$. Indeed,
\begin{align*}
\int_\Omega |\nabla_xp_\Omega(t,x,y)| dy&\ls t^{-1/\alpha}\Big(1\wedge \frac{\delta_\Omega(x)}{t^{1/\alpha}}\Big)^{\frac{\alpha-2}2}\int t^{-d/\alpha}\wedge \frac{t}{|x-y|^{d+\alpha}}dy\\
  &\ls t^{-1/\alpha}\Big(1\wedge \frac{\delta_\Omega(x)}{t^{1/\alpha}}\Big)^{\frac{\alpha-2}2}\\
  &\ls t^{-1/\alpha}q(x)\ .
\end{align*}
Let $D_j=\{x\in \Omega:\delta_\Omega(x)\approx2^{-j}\}$, $j\in\mathbb{N}$. Let $B(y,r)=\{x\in\Omega: |x-y|<r\}$, $y\in \Omega$, $r>0$. Then
\[|D_j\cap B(y,r)|\ls 2^{-j}r^{d-1}\]
since $\Omega$ is $C^1$. So we have
\begin{align*}
  \int_\Omega |\nabla_xp_\Omega(t,x,y)|q(x)dx&\ls t^{-1/\alpha}\int \Big(1\wedge \frac{\delta_\Omega(x)}{t^{1/\alpha}}\Big)^{\frac{\alpha-3}2}\Big(t^{-d/\alpha}\wedge \frac{t}{|x-y|^{d+\alpha}}\Big)dx\\
  &:= t^{-1/\alpha}(I_1+I_2)
\end{align*}
where
\begin{align*}
  I_1&=\int_{|x-y|<t^{1/\alpha}}\Big(1\wedge \frac{\delta_\Omega(x)}{t^{1/\alpha}}\Big)^{\frac{\alpha-3}2}t^{-d/\alpha}dx\\
  &\ls t^{-d/\alpha}\int_{|x-y|<t^{1/\alpha},\ \delta_\Omega(x)<t^{1/\alpha}}\Big(\frac{\delta_\Omega(x)}{t^{1/\alpha}}\Big)^{\frac{\alpha-3}2}dx+1\\
  &\ls  t^{-d/\alpha}t^{\frac{3-\alpha}{2\alpha}}\sum_{2^{-j}<t^{1/\alpha}}\int_{D_j\cap B(y,t^{1/\alpha})}2^{\frac{3-\alpha}2j}dx+1\\
  &\ls t^{-d/\alpha}t^{\frac{3-\alpha}{2\alpha}}\sum_{2^{-j}<t^{1/\alpha}}2^{\frac{3-\alpha}2j}\cdot 2^{-j}t^{\frac{d-1}\alpha}+1\\
  &\ls 1
\end{align*}
and
\begin{align*}
  I_2&=\int_{|x-y|\ge t^{1/\alpha}}\Big(1\wedge \frac{\delta_\Omega(x)}{t^{1/\alpha}}\Big)^{\frac{\alpha-3}2}\frac{t}{|x-y|^{d+\alpha}}dx\\
  &\ls \int_{|x-y|\ge t^{1/\alpha},\  \delta_\Omega(x)<t^{1/\alpha}}\Big(\frac{\delta_\Omega(x)}{t^{1/\alpha}}\Big)^{\frac{\alpha-3}2}\frac{t}{|x-y|^{d+\alpha}}dx+1\\
  &\ls \sum_{2^{-j}<t^{1/\alpha}}\sum_{2^{-k}\ge t^{1/\alpha}}\int_{D_j\cap B(y,2^{-k})}2^{\frac{3-\alpha}2j}t^{\frac{3-\alpha}{2\alpha}}\cdot t2^{k(d+\alpha)}dx+1\\
  &\ls \sum_{2^{-j}<t^{1/\alpha}}\sum_{2^{-k}\ge t^{1/\alpha}}2^{\frac{3-\alpha}2j}t^{\frac{3-\alpha}{2\alpha}}\cdot t2^{k(d+\alpha)}\cdot 2^{-j}2^{-k(d-1)}+1\\
  &\ls 1.
\end{align*}
Then the claim is proved. By Schur's test, the operator norm $\ls t^{-1/\alpha}=\lambda^\frac1\alpha$. This gives \eqref{gradu} and completes the proof.
\end{proof}
So when $1\le \alpha<2$, the commutator estimate \eqref{comm} follows from Proposition \ref{prop1} and Theorem \ref{thm1}.

\begin{remark}{\rm It is  an interesting open problem whether  \[\|e_\lambda\|_{H^{\alpha-1}}\ls \lambda^{1-\frac1\alpha}\|e_\lambda\|_2\]
still holds for $0<\alpha<1$ and $e_\lambda$ satisfying \eqref{eigen}. If it is true, then the commutator estimate \eqref{comm} follows. However, it is still possible to prove \eqref{comm} in a different way when $\alpha<1$. In particular, we will prove it when $\frac12<\alpha<1$ in the following.
}\end{remark}
\subsection{Proof when $1/2<\alpha<1$}
We need the following lemma about the nonlocal behavior of $(-\Delta)^{\alpha/2}e_\lambda$ in $\mathbb{R}^d$.
\begin{lemma}\label{lemma5}
  If $e_\lambda$ satisfies \eqref{eigen}, then for $0<\alpha<1$ we have
  \[\|(-\Delta)^{\alpha/2}e_\lambda\|_2\approx \lambda\|e_\lambda\|_2.\]
\end{lemma}
\begin{proof} On the one hand, by \eqref{eigen}\begin{align*}
    \|(-\Delta)^{\alpha/2}e_\lambda\|_2^2&\ge\int_\Omega|(-\Delta)^{\alpha/2}e_\lambda|^2= \lambda^2\|e_\lambda\|_2^2.
  \end{align*}
  On the other hand, we may write
  \begin{align*}
    \|(-\Delta)^{\alpha/2}e_\lambda\|_2^2&=\int_\Omega|(-\Delta)^{\alpha/2}e_\lambda|^2+\int_{\mathbb{R}^d\setminus\Omega}|(-\Delta)^{\alpha/2}e_\lambda|^2\\
    &\le \lambda^2\|e_\lambda\|_2^2+\int_{\mathbb{R}^d\setminus\Omega}|(-\Delta)^{\alpha/2}e_\lambda|^2
  \end{align*}
It suffices to estimate the second term. For $x\notin\Omega$,
\begin{align*}
  |(-\Delta)^{\alpha/2}e_\lambda(x)|&=c_{d,\alpha}|\int_{\mathbb{R}^d}\frac{e_\lambda(x)-e_\lambda(y)}{|x-y|^{d+\alpha}}dy|\\
  &=c_{d,\alpha}|\int_{\Omega}\frac{e_\lambda(y)}{|x-y|^{d+\alpha}}dy|\\
  &=c_{d,\alpha}|\int_\Omega\int_\Omega|x-y|^{-d-\alpha}p_\Omega(t,y,z)dye_\lambda(z)dz|\\
  &:=c_{d,\alpha}|\int_\Omega K(x,z)e_\lambda(z)dz|
\end{align*}
where $p_\Omega(t,y,z)$ is heat kernel of the Dirichlet fractional Laplacian $-(-\Delta)^{\alpha/2}|_\Omega$, and $t=\lambda^{-1}$. By Lemma \ref{heat}, we have for $0<\alpha<2$
\[p_\Omega(t,y,z)\approx \Big(1\wedge t^{-1/2}\rho(y)^{\alpha/2}\Big)\Big(1\wedge t^{-1/2}\rho(z)^{\alpha/2}\Big)\Big(t^{-d/\alpha}\wedge t|y-z|^{-d-\alpha}\Big),\]
where $y,z\in\Omega$ and $\rho(y)=dist(y,\Omega^c)$. To estimate the $L^2\to L^2$ norm of $K$, we decompose its kernel
\[K(x,z)=\int_\Omega|x-y|^{-d-\alpha}p_\Omega(t,y,z)dy=\int_{\rho(y)<t^{1/\alpha}}+\int_{\rho(y)\ge t^{1/\alpha}}:=K_1(x,z)+K_2(x,z).\]
For $K_2$, we use Young's inequality. By using $|x-y|\ge\rho(y)\ge t^{1/\alpha}$, we get
\[\int_\Omega K_2(x,z)dz\ls t^{-1}\]
\[\int_{\Omega^c} K_2(x,z)dx\ls t^{-1}\]
which gives $\|K_2\|_{2\to 2}\ls t^{-1}=\lambda$.

For $K_1$, we use Schur's test. We claim that
\[\int_\Omega K_1(x,z)dz\ls t^{-1/2}\rho(x)^{-\alpha/2}\]
\[\int_{\Omega^c} K_1(x,z)\rho(x)^{-\alpha/2}dx\ls t^{-3/2}\]
where $\rho(x)=dist(x,\Omega)$. Indeed, since $|x-y|\ge\rho(y)$ and $|x-y|\ge\rho(x)$, we have
\begin{align*}
  \int_\Omega K_1(x,z)dz&\ls \int_\Omega\int_\Omega |x-y|^{-d-\alpha}\frac{\rho(y)^{\alpha/2}}{t^{1/2}}dy\Big(t^{-d/\alpha}\wedge t|y-z|^{-d-\alpha}\Big)dz\\
  &\ls t^{-1/2}\int_\Omega |x-y|^{-d-\alpha/2}dy\\
  &\ls t^{-1/2}\rho(x)^{-\alpha/2}\ .
\end{align*}
Let $D_j=\{w\in \mathbb{R}^d:\rho(w)\approx2^{-j}\}$, $j\in\mathbb{Z}$. Note that here $w$ can be in $\Omega$ or $\Omega ^c$. For $y\in \mathbb{R}^d$, $r>0$, we denote $B(y,r)=\{x\in\mathbb{R}^d: |x-y|<r\}$. Then again we have
\[|D_j\cap B(y,r)|\ls 2^{-j}r^{d-1}\]
since $\Omega$ is $C^1$. By using $|x-y|\ge \rho(y)$ and $|x-y|\ge \rho(x)$, we get
\begin{align*}
  \int_{\Omega^c}|x-y|^{-d-\alpha}\rho(x)^{-\alpha/2}dx&\ls\sum_{2^{-k}\ge\rho(y)}\sum_{2^{-j}\le2^{-k}}\int_{D_j\cap B(y,2^{-k})}2^{k(d+\alpha)}2^{j\alpha/2}dx\\
  &\ls \sum_{2^{-k}\ge\rho(y)}\sum_{2^{-j}\le2^{-k}}2^{k(d+\alpha)}2^{j\alpha/2}\cdot 2^{-j}2^{-k(d-1)}\\
  &\ls \rho(y)^{-3\alpha/2}
\end{align*}
and then
\begin{align*}
  \int_{\Omega^c} K_1(x,z)\rho(x)^{-\alpha/2}dx&\ls t^{-1/2}\int_{\rho(y)<t^{1/\alpha}} \rho(y)^{\alpha/2}\Big(t^{-d/\alpha}\wedge \frac{t}{|y-z|^{d+\alpha}}\Big)\int_{\Omega^c}|x-y|^{-d-\alpha}\rho(x)^{-\alpha/2}dxdy\\
  &\ls t^{-1/2}\int_{\rho(y)<t^{1/\alpha}} \rho(y)^{-\alpha}\Big(t^{-d/\alpha}\wedge \frac{t}{|y-z|^{d+\alpha}}\Big)dy\\
  &:= t^{-1/2}(I_1+I_2)
\end{align*}
where
\begin{align*}
  I_1&=\int_{|y-z|<t^{1/\alpha},\ \rho(y)<t^{1/\alpha}}\rho(y)^{-\alpha}t^{-d/\alpha}dy\\
  &\ls \sum_{2^{-j}<t^{1/\alpha}}\int_{D_j\cap B(z,t^{1/\alpha})}2^{j\alpha}t^{-d/\alpha}dy\\
  &\ls \sum_{2^{-j}<t^{1/\alpha}}2^{j\alpha}t^{-d/\alpha}\cdot 2^{-j}t^{\frac{d-1}\alpha}\\
  &\ls t^{-1}
\end{align*}
and
\begin{align*}
  I_2&=\int_{|y-z|\ge t^{1/\alpha},\ {\rho(y)<t^{1/\alpha}}}\rho(y)^{-\alpha}\frac{t}{|y-z|^{d+\alpha}}dy\\
  &\ls \sum_{2^{-j}<t^{1/\alpha}}\sum_{2^{-k}\ge t^{1/\alpha}}\int_{D_j\cap B(z,2^{-k})}2^{j\alpha}\cdot t2^{k(d+\alpha)}dy\\
  &\ls  \sum_{2^{-j}<t^{1/\alpha}}\sum_{2^{-k}\ge t^{1/\alpha}}2^{j\alpha}\cdot t2^{k(d+\alpha)}\cdot 2^{-j}2^{-k(d-1)}\\
  &\ls t^{-1}\ .
\end{align*}
Hence, the claim is proved. Then we obtain $\|K_1\|_{2\to 2}\ls t^{-1}=\lambda$ by using Schur's test.

Therefore, \[\|K\|_{2\to2}\le \|K_1\|_{2\to2}+\|K_2\|_{2\to 2}\ls \lambda.\]
Then
\[\int_{\mathbb{R}^d\setminus\Omega}|(-\Delta)^{\alpha/2}e_\lambda|^2\ls \lambda^2\|e_\lambda\|_2^2.\]
So the proof is complete.
\end{proof}
Now we are ready to prove \eqref{comm} when $1/2<\alpha<1$.
\begin{theorem}\label{thm3} If $e_\lambda$ satisfies \eqref{eigen}, then for $\frac12<\alpha<1$ we have
   \[\|[\beta,(-\Delta)^{\alpha/2}]e_\lambda\|_2\ls \lambda^{1-\frac1\alpha}\|e_\lambda\|_2.\]
\end{theorem}
 \begin{proof}
   We define $Tu:=u-\lambda^{-1}(-\Delta)^{\alpha/2}u$, and write \[e_\lambda=\lambda^{-1}(-\Delta)^{\alpha/2}e_\lambda+Te_\lambda.\] One the one hand, for $\frac12\le \alpha<1$
\begin{align*}
  \|[\beta,(-\Delta)^{\alpha/2}](\lambda^{-1}(-\Delta)^{\alpha/2}e_\lambda)\|_2&\ls \lambda^{-1}\|(-\Delta)^{\alpha/2}e_\lambda\|_{H^{\alpha-1}}\\
  &\ls \lambda^{-1}\|(-\Delta)^{(\alpha-1)/2}(-\Delta)^{\alpha/2}e_\lambda\|_2\\
  &=\lambda^{-1}\|(-\Delta)^{(2\alpha-1)/2}e_\lambda\|_2\\
  &\ls\lambda^{-1}\|(-\Delta)^{\alpha/2}e_\lambda\|_2^{(2\alpha-1)/\alpha}\|e_\lambda\|_2^{(1-\alpha)/\alpha}
\end{align*}
So by Lemma \ref{lemma5}, we get
\[\|[\beta,(-\Delta)^{\alpha/2}](\lambda^{-1}(-\Delta)^{\alpha/2}e_\lambda)\|_2\ls \lambda^{(\alpha-1)/\alpha}\|e_\lambda\|_2.\]
On the other hand, observing that $Te_\lambda$ is supported on $\Omega^c$, and supp $\beta$ $\subset\subset \Omega$, we get
\begin{align*}
  |[\beta,(-\Delta)^{\alpha/2}]Te_\lambda(x)|&=|\beta(x)(-\Delta)^{\alpha/2}Te_\lambda(x)|\\
  &=c_{d,\alpha}|\beta(x)\int\frac{(Te_\lambda)(x)-(Te_\lambda)(y)}{|x-y|^{d+\alpha}}dy|\\
  &=c_{d,\alpha}|\beta(x)\int\frac{(Te_\lambda)(y)}{|x-y|^{d+\alpha}}dy|\\
  &\ls |\beta(x)|\int_{\Omega^c}|Te_\lambda(y)|dy\\
  &\le\lambda^{-1}|\beta(x)|\int_\Omega \int_{\Omega^c}\frac1{|z-y|^{d+\alpha}}dy|e_\lambda(z)| dz\\
  &\ls \lambda^{-1}|\beta(x)|\int_\Omega\rho(z)^{-\alpha}|e_\lambda(z)|dz
\end{align*}
Therefore, to prove
\[\|[\beta,(-\Delta)^{\alpha/2}]Te_\lambda\|_2\ls \lambda^{(\alpha-1)/\alpha}\|e_\lambda\|_2\]it suffices to show
\begin{equation}\label{rhoint}
  \int_\Omega\rho(z)^{-\alpha}|e_\lambda(z)|dz\ls \lambda^{(2\alpha-1)/\alpha}\|e_\lambda\|_2.
\end{equation}
Let $\delta=\frac2\alpha$. We decompose the integral into two parts: $\rho(z)<\lambda^{-\delta}$ and $\rho(z)\ge\lambda^{-\delta}$.
For the second part, we use Cauchy-Schwarz
\begin{align*}
  I_2&\le (\int_{\rho(z)\ge\lambda^{-\delta}}\rho(z)^{-2\alpha}dz)^{1/2}\|e_\lambda\|_2\\
  &\ls\lambda^{\frac{\delta}2(2\alpha-1)}\|e_\lambda\|_2\\
  &= \lambda^{(2\alpha-1)/\alpha}\|e_\lambda\|_2
\end{align*}
where we use $\alpha>1/2$ in the second inequality. For the first part, we use heat kernel
\begin{align*}
  I_1&\le\int_\Omega\int_{\rho(z)<\lambda^{-\delta}}\rho(z)^{-\alpha}p_\Omega(t,x,z)dz|e_\lambda(x)|dx\\
  &\ls \lambda^{\frac1\alpha-\delta(1-\alpha)}\|e_\lambda\|_1\\
  &\ls \lambda^{(2\alpha-1)/\alpha}\|e_\lambda\|_2\ .
\end{align*}
So we have obtained \eqref{rhoint}, and the proof  is complete.
 \end{proof}
\subsection{Some estimates when $0<\alpha\le\frac12$}
For the endpoint $\alpha=1/2$, we have \eqref{comm} with a log factor from estimating $I_2$, namely
\begin{equation}\label{comm2}
  \|[\beta,(-\Delta)^{\alpha/2}]e_\lambda\|_{2}\ls \lambda^{-1}(\log\lambda)^\frac12\|e_\lambda\|_{2}.
\end{equation}
For $0<\alpha<1/2$, we may follow the same argument to obtain
\begin{equation}\label{comm3}\|[\beta,(-\Delta)^{\alpha/2}]e_\lambda\|_{2}\ls \lambda^{-1}\|e_\lambda\|_{2}\end{equation}
which is weaker than the expected \eqref{comm}. It leads however, using the very same argument as in the previous section, to some bounds in this range. Therefore, we only state the final results for $0<\alpha\le \frac12$ in the following proposition.
\begin{proposition}\label{prop2}
  Let $d\ge1$ and $p_c=\frac{2d+2}{d-1}$. If $e_\lambda$ satisfies \eqref{eigen}, then when $\alpha=\frac12$, we have
\[\|e_\lambda\|_{L^p(K)}\ls \lambda^{(d-1)(\frac12-\frac1p)}(\log\lambda)^{\frac{d+1}4(\frac12-\frac1p)}\|e_\lambda\|_2,\ 2\le p\le p_c\]
\[\|e_\lambda\|_{L^p(K)}\ls\lambda^{2d(\frac12-\frac1p)-1}(\log\lambda)^\frac14\|e_\lambda\|_2,\ p_c\le p\le \infty.\]
Moreover, when $0<\alpha<\frac12$, we have
\[\|e_\lambda\|_{L^p(K)}\ls \lambda^{(\frac{d}{\alpha}-d-1)(\frac12-\frac1p)}\|e_\lambda\|_2,\ 2\le p\le p_c\]
\[\|e_\lambda\|_{L^p(K)}\ls\lambda^{\frac{d}{\alpha}(\frac12-\frac1p)-1}\|e_\lambda\|_2,\ p_c\le p\le \infty.\]
\end{proposition}

\section{Eigenfunction estimates when $0<\alpha<\frac12$}
In this section, we study the eigenfunction estimates when $0<\alpha<\frac12$. We may improve the estimates in Proposition \ref{prop2} when $0<\alpha<\frac12$. By interpolation, we only need to improve $L^{p_c}$ and $L^\infty$ bounds. We will decompose the range of $\alpha$ into a countable union of intervals: \[(0,\tfrac12)=\bigcup_{N\ge2} [\tfrac1{N+1},\tfrac1N)\] and establish a commutator estimate on each of them.
\begin{lemma}\label{lemma6}
  Let $N\ge2$. If $e_\lambda$ satisfies \eqref{eigen}, then for $\frac1{N+1}\le \alpha<\frac1N$ we have
   \[\|[\beta,(-\Delta)^{N\alpha/2}]e_\lambda\|_2\ls \lambda^{(N\alpha-1)/\alpha}\|e_\lambda\|_2.\]
\end{lemma}
\begin{proof} As in the proof of Theorem \ref{thm3},  we write \[e_\lambda=\lambda^{-1}(-\Delta)^{\alpha/2}e_\lambda+Te_\lambda.\] Let $N\ge2$. Then for $\frac1{N+1}\le \alpha<\frac1N$,
\begin{align*}
  \|[\beta,(-\Delta)^{N\alpha/2}](\lambda^{-1}(-\Delta)^{\alpha/2}e_\lambda)\|_2&\ls \lambda^{-1}\|(-\Delta)^{\alpha/2}e_\lambda\|_{H^{N\alpha-1}}\\
  &\ls \lambda^{-1}\|(-\Delta)^{(N\alpha-1)/2}(-\Delta)^{\alpha/2}e_\lambda\|_2\\
  &=\lambda^{-1}\|(-\Delta)^{((N+1)\alpha-1)/2}e_\lambda\|_2\\
  &\ls\lambda^{-1}\|(-\Delta)^{\alpha/2}e_\lambda\|_2^{((N+1)\alpha-1)/\alpha}\|e_\lambda\|_2^{(1-N\alpha)/\alpha}\\
  &\ls \lambda^{(N\alpha-1)/\alpha}\|e_\lambda\|_2
\end{align*}
where we use Lemma \ref{lemma5} in the last step.
On the other hand, observing that $Te_\lambda$ is supported on $\Omega^c$, and supp $\beta$ $\subset\subset \Omega$, we get
\begin{align*}
  |[\beta,(-\Delta)^{N\alpha/2}]Te_\lambda(x)|&=|\beta(x)(-\Delta)^{N\alpha/2}Te_\lambda(x)|\\
  &=c_{d,N\alpha}|\beta(x)\int\frac{(Te_\lambda)(x)-(Te_\lambda)(y)}{|x-y|^{d+N\alpha}}dy|\\
  &=c_{d,N\alpha}|\beta(x)\int\frac{(Te_\lambda)(y)}{|x-y|^{d+N\alpha}}dy|\\
  &\ls |\beta(x)|\int_{\Omega^c}|Te_\lambda(y)|dy\\
  &\le\lambda^{-1}|\beta(x)|\int_\Omega \int_{\Omega^c}\frac1{|z-y|^{d+\alpha}}dy|e_\lambda(z)| dz\\
  &\ls \lambda^{-1}|\beta(x)|\int_\Omega\rho(z)^{-\alpha}|e_\lambda(z)|dz\\
  &\ls \lambda^{-1}|\beta(x)|(\int_\Omega\rho(z)^{-2\alpha}dz)^\frac12\|e_\lambda\|_2\\
  &\ls \lambda^{-1}|\beta(x)|\|e_\lambda\|_2\ \ \ \ (using\ \alpha<\tfrac12).
\end{align*}
Thus
\[\|[\beta,(-\Delta)^{N\alpha/2}]Te_\lambda\|_2\ls \lambda^{-1}\|e_\lambda\|_2\le \lambda^{(N\alpha-1)/\alpha}\|e_\lambda\|_2\]
when $\alpha\ge \frac1{N+1}$. So the proof is complete.
\end{proof}
From the proof, we immediately have the following corollary, where $Tu=u-\lambda^{-1}(-\Delta)^{\alpha/2}u$.
\begin{corollary}\label{cor1} Let $N\ge2$ and $0<\alpha<\frac1N$. If $e_\lambda$ satisfies \eqref{eigen}, then for $1\le j\le N$
  \[\|\beta(-\Delta)^{j\alpha/2}Te_\lambda\|_2\ls \lambda^{-1}\|e_\lambda\|_2.\]
\end{corollary}

\begin{lemma}\label{lemma7}
 Let $N\ge2$ be an integer. If $e_\lambda$ satisfies \eqref{eigen}, and $v=\beta e_\lambda$, then for $0<\alpha<\frac1N$
 \[(-\Delta)^{N\alpha/2}v=\lambda^N v-\sum_{j=1}^{N-1}\lambda^{N-j}\beta(-\Delta)^{j\alpha/2}Te_\lambda-[\beta,(-\Delta)^{N\alpha/2}]e_\lambda.\]
\end{lemma}
\begin{proof} By the definition of $Te_\lambda$, we have $(-\Delta)^{\alpha/2}e_\lambda=\lambda e_\lambda-\lambda Te_\lambda$. We may use this identity repeatedly to prove the lemma. Indeed, \begin{align*}
(-\Delta)^{N\alpha/2}(\beta e_\lambda)&=\beta(-\Delta)^{N\alpha/2}e_\lambda-[\beta,(-\Delta)^{N\alpha/2}]e_\lambda\\
&=\beta(-\Delta)^{(N-1)\alpha/2}(\lambda e_\lambda-\lambda Te_\lambda)-[\beta,(-\Delta)^{N\alpha/2}]e_\lambda\\
&=\lambda\beta(-\Delta)^{(N-1)\alpha/2}e_\lambda-\lambda\beta(-\Delta)^{(N-1)\alpha/2} Te_\lambda-[\beta,(-\Delta)^{N\alpha/2}]e_\lambda\\
&=...\\
&=\lambda^{N-1}\beta(-\Delta)^{\alpha/2}e_\lambda-\sum_{j=1}^{N-1}\lambda^{N-j}\beta(-\Delta)^{j\alpha/2}Te_\lambda-[\beta,(-\Delta)^{N\alpha/2}]e_\lambda\\
&=\lambda^N\beta e_\lambda-\sum_{j=1}^{N-1}\lambda^{N-j}\beta(-\Delta)^{j\alpha/2}Te_\lambda-[\beta,(-\Delta)^{N\alpha/2}]e_\lambda
\end{align*}
where we use $(-\Delta)^{\alpha/2}e_\lambda=\lambda e_\lambda$ in $\Omega$ and supp $\beta$ $\subset \Omega$.
\end{proof}
\subsection{$L^{p_c}$ bounds}
\begin{theorem}\label{thm4}Let $d\ge1$, $0<\alpha<\frac12$, and $p_c=\frac{2d+2}{d-1}$. If $e_\lambda$ satisfies \eqref{eigen}, then for $\lambda>1$
\[\|e_\lambda\|_{L^{p_c}(K)}\ls \lambda^{\frac{d(3d+1)}{2(d+1)(2d+1)}\frac1\alpha-\frac{d+1}{2d+1}}\|e_\lambda\|_2.\]
\end{theorem}
Recall that Proposition \ref{prop2} gives for $0<\alpha<\frac12$,
\[\|e_\lambda\|_{L^{p_c}(K)}\ls \lambda^{\frac{d}{d+1}\frac1\alpha-1}\|e_\lambda\|_2.\]
Theorem \ref{thm4} gives us better bounds, since
\[\frac{d(3d+1)}{2(d+1)(2d+1)}\frac1\alpha-\frac{d+1}{2d+1}<\frac{d}{d+1}\frac1\alpha-1\iff \alpha<\frac12.\]
However, it is still weaker than the conjecture bound
\[\|e_\lambda\|_{L^{p_c}(K)}\ls \lambda^{\frac{d-1}{2d+2}\frac1\alpha}\|e_\lambda\|_2\]
since
\[\frac{d(3d+1)}{2(d+1)(2d+1)}\frac1\alpha-\frac{d+1}{2d+1}>\frac{d-1}{2d+2}\frac1\alpha \iff \alpha<\frac12.\]
Furthermore, one may observe that it is exactly a linear combination of the two bounds:
\[\frac{d(3d+1)}{2(d+1)(2d+1)}\frac1\alpha-\frac{d+1}{2d+1}=\theta\Big(\frac{d}{d+1}\frac1\alpha-1\Big)+(1-\theta)\Big(\frac{d-1}{2d+2}\frac1\alpha\Big)\]
where $\theta=\frac{d+1}{2d+1}$.

\begin{proof}
 It suffices to prove the estimate for $\frac1{N+1}\le \alpha<\frac1N$, $N\ge2$. Fix $z=\lambda^N+\lambda^{N-\delta}i$. Here $\delta>0$ will be determined later. On the one hand, by Lemma \ref{lemma7}, we have
 \[v=((-\Delta)^{N\alpha/2}-z)^{-1}((\lambda^N-z)v+\sum_{j=1}^{N-1}\lambda^{N-j}\beta(-\Delta)^{j\alpha/2}Te_\lambda-[\beta,(-\Delta)^{N\alpha/2}]e_\lambda).\]
 As in Section 2, we write \[((-\Delta)^{N\alpha/2}-z)^{-1}=T_0+R\]
 where $T_0=\frac{z^{(2-N\alpha)/(N\alpha)}}{N\alpha/2}(-\Delta-z^{2/(N\alpha)})^{-1}$. By \eqref{uest}, we have
 \[\|e_\lambda\|_{L^{p_c}(K)}\ls \|e_\lambda\|_2+\|Qv\|_{p_c}.\]
Then by using Lemma \ref{lemma1} and the estimates of $R$ and $Q$ in Section 2, we have
 \[\|QT_0\|_{2\to p_c}\ls \|Q\|_{p_c\to p_c}\|T_0\|_{2\to p_c}\ls\lambda^{\frac{\delta}2+\frac{d}{d+1}\frac1\alpha-N}\]
 \[\|QT_0\|_{p_c\to p_c}\ls \|Q\|_{p_c\to p_c}\|T_0\|_{p_c\to p_c}\ls\lambda^{\frac{3d+1}{2d+2}\delta-N}\]
 \[\|QR\|_{2\to p_c}\ls \lambda^{\frac{d}{d+1}\frac1\alpha-N}\]
  \[\|QR\|_{p_c\to p_c}\ls \lambda^{-N}.\]
 On the other hand, by Lemma \ref{lemma6} and Corollary \ref{cor1}
 \[\|(\lambda^N-z)v\|_2\ls \lambda^{N-\delta}\|e_\lambda\|_2\]
 \[\|\lambda^{N-j}\beta(-\Delta)^{j\alpha/2}Te_\lambda\|_{p_c}\ls \lambda^{N-j-1}\|e_\lambda\|_2,\ \ 1\le j\le N\]
 \[\|[\beta,(-\Delta)^{N\alpha/2}]e_\lambda\|_2\ls \lambda^{(N\alpha-1)/\alpha}\|e_\lambda\|_2.\]
 Therefore,
 \[\|(QT_0+QR)((\lambda^N-z)v)\|_{p_c}\ls \lambda^{\frac{d}{d+1}\frac1\alpha-\frac{\delta}2}\|e_\lambda\|_2\]
 \[\|(QT_0+QR)(\lambda^{N-j}\beta(-\Delta)^{j\alpha/2}Te_\lambda)\|_{p_c}\ls \lambda^{\frac{3d+1}{2d+2}\delta-2}\|e_\lambda\|_2, \ \ 1\le j\le N\]
 \[\|(QT_0+QR)([\beta,(-\Delta)^{N\alpha/2}]e_\lambda)\|_{p_c}\ls \lambda^{\frac{\delta}2-\frac1{d+1}\frac1\alpha}\|e_\lambda\|_2.\]
 Then it is straightforward to find the optimal choice of $\delta$ to minimize
 \[\max\Big\{\frac{d}{d+1}\frac1\alpha-\frac{\delta}2,\ \frac{3d+1}{2d+2}\delta-2,\ \frac{\delta}2-\frac1{d+1}\frac1\alpha\Big\}.\]If we set $\delta=\frac{d+\alpha(2d+2)}{\alpha(2d+1)}$, which is from solving  $\frac{d}{d+1}\frac1\alpha-\frac{\delta}2=\frac{3d+1}{2d+2}\delta-2$, then the theorem follows.
\end{proof}
\subsection{$L^\infty$ bounds}
\begin{theorem}\label{thm5}Let $d\ge1$ and $0<\alpha<\frac12$. If $e_\lambda$ satisfies \eqref{eigen}, then for $\lambda>1$
\[\|e_\lambda\|_{L^{\infty}(K)}\ls \lambda^{\frac{d(d+1)-4\alpha}{2\alpha(d+2)}}\|e_\lambda\|_2.\]
\end{theorem}
Proposition \ref{prop2} gives for $0<\alpha<\frac12$,
\[\|e_\lambda\|_{L^{\infty}(K)}\ls \lambda^{\frac{d}{2\alpha}-1}\|e_\lambda\|_2.\]
Theorem \ref{thm5} gives us  better bounds, since
\[\frac{d(d+1)-4\alpha}{2\alpha(d+2)}<\frac{d}{2\alpha}-1\iff \alpha<\frac12.\]
However, it is still weaker than the conjecture bound
\[\|e_\lambda\|_{L^{p_c}(K)}\ls \lambda^{\frac{d-1}{2\alpha}}\|e_\lambda\|_2\]
since
\[\frac{d(d+1)-4\alpha}{2\alpha(d+2)}>\frac{d-1}{2\alpha} \iff \alpha<\frac12.\]
Furthermore, one may observe that it is exactly a linear combination of the two bounds:
\[\frac{d(d+1)-4\alpha}{2\alpha(d+2)}=\theta\Big(\frac{d}{2\alpha}-1\Big)+(1-\theta)\Big(\frac{d-1}{2\alpha}\Big)\]
where $\theta=\frac{2}{d+2}$. As $d\to \infty$, we see that it approaches the conjecture bound.
\begin{proof}
  As before, it suffices to prove the estimate for $\frac1{N+1}\le \alpha<\frac1N$, $N\ge2$. Fix $z=\lambda^N+\lambda^{N-\delta}i$. Here $\delta>0$ will be determined later. On the one hand, by Lemma \ref{lemma7}, we have
 \[v=((-\Delta)^{N\alpha/2}-z)^{-1}((\lambda^N-z)v+\sum_{j=1}^{N-1}\lambda^{N-j}\beta(-\Delta)^{j\alpha/2}Te_\lambda-[\beta,(-\Delta)^{N\alpha/2}]e_\lambda).\]
 As in Section 2, we write \[((-\Delta)^{N\alpha/2}-z)^{-1}=T_0+R\]
 where $T_0=\frac{z^{(2-N\alpha)/(N\alpha)}}{N\alpha/2}(-\Delta-z^{2/(N\alpha)})^{-1}$.
 By \eqref{uest}, we have
 \[\|e_\lambda\|_{L^\infty(K)}\ls \|e_\lambda\|_2+\|Qv\|_\infty.\]
Then by applying Lemma \ref{lemma1} and the estimates of $R$ and $Q$ in Section 2, we obtain
 \[\|QT_0\|_{2\to \infty}\ls \|Q\|_{p_c\to\infty}\|T_0\|_{2\to p_c}\ls \lambda^{\frac{\delta}2+\frac{d}{2\alpha}-N}\]
 \[\|QT_0\|_{\infty\to \infty}\ls \|Q\|_{\infty\to\infty}\|T_0\|_{\infty\to \infty}\ls \lambda^{\frac{d+1}{2}\delta-N}\]
 \[\|QR\|_{2\to \infty}\ls \lambda^{\frac{d}{2\alpha}-N}\]
  \[\|QR\|_{\infty\to \infty}\ls \lambda^{-N}.\]
 On the other hand, by Lemma \ref{lemma6} and Corollary \ref{cor1}
 \[\|(\lambda^N-z)v\|_2\ls \lambda^{N-\delta}\|e_\lambda\|_2\]
 \[\|\lambda^{N-j}\beta(-\Delta)^{j\alpha/2}Te_\lambda\|_{\infty}\ls \lambda^{N-j-1}\|e_\lambda\|_2,\ \ 1\le j\le N\]
 \[\|[\beta,(-\Delta)^{N\alpha/2}]e_\lambda\|_2\ls \lambda^{(N\alpha-1)/\alpha}\|e_\lambda\|_2.\]
 Therefore,
 \[\|(QT_0+QR)((\lambda^N-z)v)\|_{\infty}\ls \lambda^{\frac{d}{2\alpha}-\frac{\delta}2}\|e_\lambda\|_2\]
 \[\|(QT_0+QR)(\lambda^{N-j}\beta(-\Delta)^{j\alpha/2}Te_\lambda)\|_{\infty}\ls \lambda^{\frac{d+1}{2}\delta-2}\|e_\lambda\|_2, \ \ 1\le j\le N\]
 \[\|(QT_0+QR)([\beta,(-\Delta)^{N\alpha/2}]e_\lambda)\|_{\infty}\ls \lambda^{\frac{d-2}{2\alpha}+\frac{\delta}2}\|e_\lambda\|_2.\]
 Now we need to find the optimal choice of $\delta$ to minimize
 \[\max\Big\{\frac{d}{2\alpha}-\frac{\delta}2,\ \frac{d+1}{2}\delta-2,\ \frac{d-2}{2\alpha}+\frac{\delta}2\Big\}.\]If we set $\delta=\frac{d+4\alpha}{\alpha(d+2)}$, which is from solving  $\frac{d}{2\alpha}-\frac{\delta}2=\frac{d+1}{2}\delta-2$, then the theorem follows.
\end{proof}

\section{ One dimensional estimates when $0<\alpha\le \frac12$}
In this section, we study one dimensional eigenfunction estimates. We have already proved the uniform bounds for the $L^\infty(K)$ norm of the eigenfunctions when $\frac12<\alpha<2$ by using the commutator estimate \eqref{comm} in the one dimensional case. So we only need to consider $\alpha\le \frac12$ in this section. Note that $p_c=\infty$ when $d=1$.

By Proposition \ref{prop2}, we have for $\alpha=\frac12$,
\begin{equation}\label{d11}
  \|e_\lambda\|_{L^p(K)}\ls (\log\lambda)^{\frac14-\frac1{2p}}\|e_\lambda\|_2,\ 2\le p\le \infty.
\end{equation}
By Theorem \ref{thm5} and interpolation, we have for $0<\alpha<\frac12$
\begin{equation}\label{d12}
  \|e_\lambda\|_{L^p(K)}\ls \lambda^{\frac{1-2\alpha}{3\alpha}(1-\frac2p)}\|e_\lambda\|_2,\ 2\le p\le \infty.
\end{equation}
Recall that Kwasnicki \cite{kwasnicki}, Proposition 2 proved uniform $L^\infty$ bounds when $\frac12\le \alpha<2$ by constructing approximate eigenfunctions. See also \cite{stos} for the case $\alpha=1$. If we combine the resolvent estimates with the approximation method, we may prove uniform bounds in a larger range of $\alpha$. Indeed, we will prove uniform bounds for the $L^\infty(K)$ norm when $\frac14\le \alpha<2$, and the $L^p(K)$ norm when $p\le \frac{2(1-2\alpha)}{1-4\alpha}$ and $0<\alpha<\frac14$.

\subsection{Approximation results}
Let $0<\alpha\le \frac12$, $\Omega=(-1,1)$, $n\ge1$. Let $e_{\lambda_n}$ be the  eigenfunctions of $(-\Delta)^{\alpha/2}|_\Omega$ associated with eigenvalues $\lambda_n$, and $\|e_{\lambda_n}\|_2=1$. Let $\tilde e_{\lambda_n}$ be the approximate eigenfunctions defined by \cite{kwasnicki}, Equation $(13)$  associated with $\mu_n\approx n$. By the definition of $\tilde e_{\lambda_n}$, we have
\begin{equation}\label{phitilde}
  \|\tilde e_{\lambda_n}\|_\infty\ls 1.
\end{equation}
We need the following approximation results from \cite{kwasnicki}.

\begin{lemma}[\cite{kwasnicki}, Theorem 1]For large $n$ (i.e. $n\ge N(\alpha)$ for some $N(\alpha)>0$)
\begin{equation}\label{eigenvalue}
  |\lambda_n-\mu_n^\alpha|\ls \frac1n\approx \lambda_n^{-1/\alpha}.
\end{equation}
\end{lemma}
 \begin{lemma}[\cite{kwasnicki}, Lemma 1, Equation (22)]
   \begin{equation}\label{apperr}
\|(-\Delta)^{\alpha/2}\tilde e_{\lambda_n}-\mu_n^\alpha\tilde e_{\lambda_n}\|_{L^\infty(\Omega)}\ls \frac1n\approx \lambda_n^{-1/\alpha}.
\end{equation}
 \end{lemma}
It is stronger than the $L^2$ estimate stated there, but the $L^\infty$ estimate follows directly from the equation (22) in the proof.

\begin{lemma}[\cite{kwasnicki}, Proposition 1] For large $n$, we have
  \begin{equation}\label{L2err}
 \|e_{\lambda_n}-\tilde e_{\lambda_n}\|_2\ls \frac1{n^\alpha}\approx \lambda_n^{-1}.
\end{equation}
\end{lemma}

\subsection{$L^\infty$ bounds}
\begin{theorem} Let  $\Omega=(-1,1)$. Let $K\subset\subset \Omega$ be a compact set. Then for $\lambda>1$ we have
\begin{equation}\label{Linftybd0}\|e_{\lambda}\|_{L^\infty(K)}\ls \|e_{\lambda}\|_2,\ \frac14\le \alpha<2\end{equation}
\begin{equation}\label{Linftybd1}
  \|e_{\lambda}\|_{L^\infty(K)}\ls \lambda^{\frac1{2\alpha}-2}\|e_{\lambda}\|_2,\ \frac18\le \alpha<\frac14
\end{equation}
\begin{equation}\label{Linftybd2}
  \|e_{\lambda}\|_{L^\infty(K)}\ls \lambda^{\frac1{3\alpha}-\frac23}\|e_{\lambda}\|_2,\ 0< \alpha<\frac18.
\end{equation}
The constant may depend on $\alpha$ and $K$, but it is independent of $\lambda$.
\end{theorem}

Here \eqref{Linftybd2} follows from \eqref{d12}. Note that \[\frac1{2\alpha}-2>\frac1{3\alpha}-\frac23 \iff \alpha<\frac18.\] So \eqref{Linftybd0} and \eqref{Linftybd1} give better bounds when $\alpha\ge\frac18$.

\begin{proof}We only need to consider $\alpha\le \frac12$. Let $u:=e_{\lambda_n}-\tilde e_{\lambda_n}$. By using \eqref{L2err} we have
\[\|u\|_2\ls \lambda_n^{-1}.\]It suffices to estimate $\|u\|_{L^\infty(K)}$ for $K\subset\subset \Omega$. Let $v=\beta  u$, $v=u$ on $K_1\supset\supset K$.  Here $\beta$ is the same smooth cut-off function defined in Section 2. For large $n$,
\begin{align*}(-\Delta)^{\alpha/2}v&=\beta(-\Delta)^{\alpha/2}(e_{\lambda_n}-\tilde e_{\lambda_n})+[\beta,(-\Delta)^{\alpha/2}](e_{\lambda_n}-\tilde e_{\lambda_n})\\
&=\lambda_n\beta e_{\lambda_n}-\mu_n^\alpha\beta\tilde e_{\lambda_n}+[\beta,(-\Delta)^{\alpha/2}](e_{\lambda_n}-\tilde e_{\lambda_n})+R_0\\
&=\lambda_n\beta(e_{\lambda_n}-\tilde e_{\lambda_n})+[\beta,(-\Delta)^{\alpha/2}](e_{\lambda_n}-\tilde e_{\lambda_n})+R_1\\
&=\lambda_n v+[\beta,(-\Delta)^{\alpha/2}]u+R_1
\end{align*}
where $|R_0|\ls \frac1n$ by \eqref{apperr}, and $|R_1|\ls \frac1n$ by using \eqref{phitilde} and \eqref{eigenvalue}. Fix  $z=\lambda_n+\lambda_n^{-1}i$. Thus
\begin{equation}\label{inverse}
  v=((-\Delta)^{\alpha/2}-z)^{-1}((\lambda_n-z)v+[\beta,(-\Delta)^{\alpha/2}]u+R_1)
\end{equation}

Then for $x\in K\subset\subset K_1$ and $t=\lambda_n^{-1}$
\begin{align*}
  e^{-1}|u(x)|&=|P_t^{\Omega}u(x)+R_2(x)|\le |\int_{K_1}p_\Omega(t,x,y)u(y)dy|+|\int_{\Omega\setminus K_1}p_\Omega(t,x,y)u(y)dy|+ \|R_2\|_\infty\\
  &\ls |\int_{K_1}p_\Omega(t,x,y)v(y)dy|+\|u\|_2+\|R_2\|_\infty\\
  :&=|Qv(x)|+1
\end{align*}
where $R_2(x)=P_t^\Omega \tilde e_{\lambda_n}(x)-\tilde e_{\lambda_n}(x)$ satisfies $\|R_2\|_\infty\ls 1$. As before,  we may write $v=T_0w+Rw$, where $w=(\lambda_n-z)v+[\beta,(-\Delta)^{\alpha/2}]u+R_1$.
When $z=\lambda_n+\lambda_n^{-1}i$, we have
\[\|T_0\|_{2\to \infty}\ls |z|^{\frac{1}{2\alpha}-1}dist((\tfrac{z}{|z|})^{2/\alpha},[0,\infty))^{-\frac12}\approx \lambda_n^{\frac{1}{2\alpha}}\]
Hence
\[\|QT_0\|_{2\to\infty}\ls \lambda_n^{\frac{1}{2\alpha}}\]
Moreover, by \eqref{QR}
\[\|QR\|_{2\to\infty}\ls \lambda_n^{\frac1{2\alpha}-1}.\]
Then using \eqref{eigenvalue}--\eqref{L2err}, we claim that for $\alpha<\frac12$
\begin{equation}\label{claim}
  \|[\beta,(-\Delta)^{\alpha/2}]u\|_2\ls \lambda_n^{-2}.
\end{equation}
Indeed, we may write $u=\lambda_n^{-1}(-\Delta)^{\alpha/2}u+T_1u+T_2u$ where
\[T_1u=u-\lambda_n^{-1}\textbf{1}_{\Omega}(-\Delta)^{\alpha/2}u\]
\[T_2u=-\lambda_n^{-1}\textbf{1}_{\Omega^c}(-\Delta)^{\alpha/2}u.\]
Then for $\alpha<\frac12$, we have
\begin{equation*}
\|[\beta,(-\Delta)^{\alpha/2}](\lambda_n^{-1}(-\Delta)^{\alpha/2}u)\|_2\ls \lambda_n^{-1}\|(-\Delta)^{\alpha/2}u\|_{H^{\alpha-1}}\ls \lambda_n^{-1}\|u\|_{H^{2\alpha-1}}\ls \lambda_n^{-1}\|u\|_2\ls \lambda_n^{-2}
\end{equation*}
and
\begin{align*}
\|[\beta,(-\Delta)^{\alpha/2}](T_1u)\|_2\ls\|T_1u\|_{H^{\alpha-1}}\le\|T_1u\|_{L^2(\Omega)}&=\lambda_n^{-1}\|(-\Delta)^{\alpha/2}\tilde\varphi_n-\lambda_n\tilde\varphi_n\|_{L^2(\Omega)}\\
&\ls\lambda_n^{-1}\lambda_n^{-1/\alpha}\le\lambda_n^{-2}\end{align*}
Observing that $T_2u$ is supported on $\Omega^c$, and supp $\beta$ $\subset\subset \Omega$, we get
\begin{align*}
  |[\beta,(-\Delta)^{\alpha/2}](T_2u)(x)|&=|\beta(x)(-\Delta)^{\alpha/2}T_2u(x)|\\
  &=c_{1,\alpha}|\beta(x)\int\frac{(T_2u)(x)-(T_2u)(y)}{|x-y|^{1+\alpha}}dy|\\
  &=c_{1,\alpha}|\beta(x)\int\frac{(T_2u)(y)}{|x-y|^{1+\alpha}}dy|\\
  &\ls |\beta(x)|\int_{\Omega^c}|T_2u(y)|dy\\
  &\le\lambda_n^{-1}|\beta(x)|\int_\Omega \int_{\Omega^c}\frac1{|z-y|^{1+\alpha}}dy|u(z)| dz\\
  &\ls \lambda_n^{-1}|\beta(x)|\int_\Omega\rho(z)^{-\alpha}|u(z)|dz\\
  &\ls \lambda_n^{-1}|\beta(x)|(\int_\Omega\rho(z)^{-2\alpha}dz)^\frac12\|u\|_2\\
  &\ls \lambda_n^{-2}|\beta(x)|\ \ \ \ (using \ \alpha<\tfrac12\ and\ \eqref{L2err})
\end{align*}
So
\[\|[\beta,(-\Delta)^{\alpha/2}](T_2u)\|_2\ls \lambda_n^{-2}.\]
Then the claim \eqref{claim} is proved.
Then we get for $\alpha<\frac12$
\[\|w\|_2\ls \lambda_n^{-2}\]
since $|R_1|\ls \frac1n\approx \lambda_n^{-1/\alpha}\le \lambda_n^{-2}$. Then for $\alpha<\frac12$
\[\|u\|_{L^\infty(K)}\ls 1+(\lambda_n^{\frac1{2\alpha}}+\lambda_n^{\frac1{2\alpha}-1})\|w\|_2\ls 1+\lambda^{\frac1{2\alpha}-2}\]
which is uniformly bounded when $\frac14\le \alpha<\frac12$.

When $\alpha=\frac12$,  the uniform bound can be obtained by using the trivial commutator estimate
\[\|[\beta,(-\Delta)^{\alpha/2}]u\|_2\ls \|u\|_{H^{\alpha-1}}\ls \|u\|_2\ls\lambda_n^{-1}\]
and setting $z=\lambda_n+i$. Then
\[\|w\|_2\ls \lambda_n^{-1}\]
and when $\alpha=\frac12$
\[\|u\|_{L^\infty(K)}\ls 1+(\lambda_n^{\frac1{2\alpha}}+\lambda_n^{\frac1{2\alpha}-1})\|w\|_2\ls 1+\lambda_n^{\frac1{2\alpha}-1}\approx1.\]

Consequently, for $\frac14\le \alpha<2$ we have the uniform bound
\[\|e_{\lambda_n}\|_{L^\infty(K)}\ls 1.\]
Moreover, for $0<\alpha<\frac14$  we have
\[\|e_{\lambda_n}\|_{L^\infty(K)}\ls \lambda_n^{\frac1{2\alpha}-2}.\]
\end{proof}
\subsection{$L^p$ bounds}

\begin{theorem} Let $0< \alpha<\frac14$, $\Omega=(-1,1)$, and $p_1=\frac{2(1-2\alpha)}{1-4\alpha}$. Let $K\subset\subset \Omega$ be a compact set. Then for $\lambda>1$ we have
\[\|e_{\lambda}\|_{L^p(K)}\ls \|e_{\lambda}\|_2,\ 2\le p\le p_1.\]
Moreover, when $\frac18\le \alpha<\frac14$,
\[\|e_{\lambda}\|_{L^p(K)}\ls \lambda^{\frac{1-4\alpha}{2\alpha}(1-\frac{p_1}{p})}\|e_{\lambda}\|_2,\  p_1<p\le\infty.\]
When $0<\alpha<\frac18$,
  \[\|e_{\lambda}\|_{L^p(K)}\ls \lambda^{\frac{1-2\alpha}{3\alpha}(1-\frac{p_1}{p})}\|e_{\lambda}\|_2,\  p_1<p\le\infty.\]
The constants may depend on $\alpha$ and $K$, but it is independent of $\lambda$.
\end{theorem}
This theorem gives better bounds than \eqref{d12}. The idea is essentially the same as before. We only sketch the proof here. Let $p_o=\frac2{1-2\alpha}$. Fix $z=\lambda_n+\lambda_n^{-1}i$. Then
\[\|T_0\|_{2\to p}\ls |z|^{\frac{1}{\alpha}(\frac12-\frac1p)-1}dist((\tfrac{z}{|z|})^{2/\alpha},[0,\infty))^{-\frac12-\frac1p}\approx \lambda_n^{\frac1\alpha(\frac12-\frac1p)+\frac2p}\]
Then
\[\|QT_0\|_{2\to p}\ls \lambda_n^{\frac1\alpha(\frac12-\frac1p)+\frac2p}.\]
Moreover, by \eqref{QR}
\[\|QR\|_{2\to p}\ls \lambda_n^{\frac{1}{\alpha}(\frac12-\frac1p)-1}.\]
 Using \eqref{claim}, we get
\[\|u\|_{L^p(K)}\ls 1+(\lambda_n^{\frac1\alpha(\frac12-\frac1p)+\frac2p}+\lambda_n^{\frac{1}{\alpha}(\frac12-\frac1p)-1})\|w\|_2\ls 1+\lambda_n^{\frac1\alpha(\frac12-\frac1p)+\frac2p-2},\ 2\le p\le\infty\]
which is uniformly bounded when $p\le  p_1:=\frac{2(1-2\alpha)}{1-4\alpha}$. Interpolating the $L^{p_1}$ bound with the $L^\infty$ bound in \eqref{Linftybd1} (or \eqref{Linftybd2}), we obtain the $L^p$ bounds in the theorem.

\begin{remark}{\rm It is natural to ask whether the commutator estimate \eqref{claim} can be improved. If one can prove for $\alpha<1/4$
\[ \|[\beta,(-\Delta)^{\alpha/2}]u\|_2\ls \lambda_n^{-\frac1{2\alpha}}\]
then by setting $z=\lambda_n+\lambda_n^{1-\frac1{2\alpha}}i$ one can get
\[\|w\|_2\ls \lambda_n^{-\frac1{2\alpha}}\] which implies the uniform bound for $\|e_\lambda\|_{L^\infty(K)}$ when $0<\alpha<1/4$. But in the method described above, it seems difficult to improve the estimates $$\|[\beta,(-\Delta)^{\alpha/2}](\lambda_n^{-1}(-\Delta)^{\alpha/2}u)\|_2\ls \lambda_n^{-2},$$
$$\|[\beta,(-\Delta)^{\alpha/2}](T_2u)\|_2\ls \lambda_n^{-2},$$ due to the approximation error estimate \eqref{L2err} for $\|u\|_2=\|e_{\lambda_n}-\tilde e_{\lambda_n}\|_2$. Thus it suffices to improve \eqref{L2err} to
\begin{equation}\label{besterr}
  \|e_{\lambda_n}-\tilde e_{\lambda_n}\|_2\ls \frac1{\sqrt{n}}\approx \lambda_n^{-\frac1{2\alpha}}.
\end{equation}
Note that it has been proved for $\alpha\ge1/2$ in Proposition 1 in \cite{kwasnicki}, but unknown for $\alpha<1/2$.
}\end{remark}

\section*{Acknowledgments}
Y.S. is partly supported by the Simons Foundation.

\bibliography{biblio}
\bibliographystyle{alpha}
\medskip

\small

\begin{center}
 Department of Mathematics \\
Johns Hopkins University \\
3400 N. Charles Street \\
Baltimore, MD 21218\\
xhuang49@jhu.edu,\,\,\,sire@math.jhu.edu,\,\,\,czhang67@jhu.edu
\end{center}

\end{document}